\documentclass[12pt,a4paper]{amsart}
\usepackage{amssymb,amsmath}
\usepackage{hyperref}

\numberwithin{equation}{section}
\newtheorem{theorem}{Theorem}[section]
\newtheorem{lemma}[theorem]{Lemma}
\newtheorem{proposition}[theorem]{Proposition}
\newtheorem{corollary}[theorem]{Corollary}

\theoremstyle{definition}
\newtheorem{definition}[theorem]{Definition}

\theoremstyle{remark}
\newtheorem{remark}[theorem]{Remark}
\newtheorem{notation}[theorem]{Notation}

\newcommand\Supp{\operatorname{Supp}}

\newcommand\Tor{\operatorname{Tor}}
\newcommand\Hom{\operatorname{Hom}}

\newcommand\Ext{\operatorname{Ext}}
\newcommand\Rad{\operatorname{Rad}}

\newcommand\im{\operatorname{Im}}

\newcommand\inj{\operatorname{inj}}
\newcommand\depth{\operatorname{depth}}

\newcommand\grade{\operatorname{grade}}

\newcommand{\End}{\operatorname{End}}
\newcommand{\qism}{\stackrel{\sim}{\longrightarrow}}

\begin{document}
\author[W. Mahmood]{Waqas Mahmood}
\title[Natural homomorphisms of local cohomology]{On natural homomorphisms of local cohomology modules}

\address{Quaid-I-Azam University Islamabad, Lahore-Pakistan}%
\email{ waqassms$@$gmail.com}

\thanks{This research was partially supported by the Higher Education Commission, Pakistan}


\begin{abstract}
Let $M$ be a non-zero finitely generated module over a finite dimensional commutative Noetherian local ring $(R,\mathfrak{m})$ with $\dim_R(M)=t$. Let $I$ be an ideal of $R$ with $\grade(I,M)=c$. In this article we will investigate several natural homomorphisms of local cohomology modules. The main purpose of this article is to investigate that the natural homomorphisms $\Tor^{R}_c(k,H^c_I(M))\to k\otimes_R M$ and $\Ext^{d}_R(k,H^c_I(M))\to  \Ext^{t}_R(k, M)$ are non-zero where $d:=t-c$. In fact for a Cohen-Macaulay module $M$ we will show that the homomorphism $\Ext^{d}_R(k,H^c_I(M))\to \Ext^{t}_R(k, M)$ is injective (resp. surjective) if and only if the homomorphism $H^{d}_{\mathfrak{m}}(H^c_{I}(M))\to H^t_{\mathfrak{m}}(M)$ is injective (resp. surjective) under the additional assumption of vanishing of Ext modules. The similar results are obtained for the homomorphism $\Tor^{R}_c(k,H^c_I(M))\to k\otimes_R M$. Moreover we will construct the natural homomorphism $\Tor^{R}_c(k, H^c_I(M))\to \Tor^{R}_c(k, H^c_J(M))$ for the ideals $J\subseteq I$ with $c = \grade(I,M)= \grade(J,M)$. There are several sufficient conditions on $I$ and $J$ to prove this homomorphism is an isomorphism.
\end{abstract}

\maketitle

\section{Introduction}
Let $I$ denote an ideal of a Noetherian local ring $(R,\mathfrak{m})$. We denote $H^i_I(R), i\in \mathbb{Z},$ the local cohomology modules of $R$ with respect to $I$. For the definition of local cohomology  we refer to \cite{goth} and \cite{b}. During the last few years many authors have investigated the natural homomorphism
$\mu: R\to \End_R(H^c_{I}(R))$. Firstly, in 2008, M. Hellus and J. St$\ddot{u}$ckrad (see \cite[Theorem 2.2]{he1}) have shown that this natural homomorphism is an isomorphism for a complete local ring and a cohomologically complete intersection ideal $I$. After that several authors generalized this idea to an arbitrary module.

Most recently the extension of this natural homomorphism was proven by the author and Z. Zahid to a finitely generated $R$-module and the canonical module (see \cite[Theorem 1.1]{waqas1} and \cite[Theorem 1.1]{z}). Moreover, for a complete local Gorenstein ring, in \cite[Theorems 1.1 and 4.4]{waqas} the author and P. Schenzel have discussed the injectivity and surjectivity of the homomorphism $\mu$ in terms of numerical invariants $\tau_{d,d}$, where $\tau_{i,j}$ is defined to be the socle dimension of $H^{i}_{\mathfrak{m}}(H^{n-j}_{I}(R))$. M. Hellus and P. Schenzel, in \cite[Conjecture 2.7]{pet1}, conjectured that the natural homomorphism $\Ext^d_R(k, H^c_I(R))\to k$ is non-zero. As a first step it was proven (see \cite[Theorem 6.2]{p5}) to be true for a regular local ring containing a field. In \cite[Theorem 4.4]{waqas} there is a relation between this conjecture and the above homomorphism $\mu$. One of our main results is a contribution to check the validity of this conjecture for Cohen-Macaulay modules. Moreover we are succeeded to characterize several interpretations such that the natural homomorphism $\Tor^{R}_c(k,H^c_I(M))\to k\otimes_R M$ is non-zero. In particular we prove the following result:

\begin{theorem}
Let $M$ be a non-zero finitely generated $R$-module of $\dim_R(M)= t$. Suppose that $\grade(I,M)= c$ for an ideal $I$. Then
\begin{itemize}
\item[(a)] The following conditions are equivalent:
\begin{itemize}
\item[(1)] The natural map $\Tor^{R}_c(k,H^c_I(M))\to k\otimes_R M$ is surjective and $\Tor_i^R(k,H^c_I(M))=0$ for all $i< c$.
\item[(2)] The natural map $\Hom_R(M,k)\to H^c_{\mathfrak m}(D(H^c_{I}(M)))$ is injective and $H^i_{\mathfrak m} (D(H^c_I(M)))= 0$ for all $i<  c$.
\item[(3)] The natural map $D(M)\to H^c_{\mathfrak m}(D(H^c_{I}(M)))$ is injective and $H^i_{\mathfrak m} (D(H^c_I(M)))= 0$ for all $i<  c$.
\end{itemize}
\item[(b)] If $M$ is Cohen-Macaulay and $d: = t-c$. Then the following conditions are equivalent:
\begin{itemize}
\item[(1)] The natural map $H^{d}_{\mathfrak{m}}(H^c_{I}(M))\to H^t_{\mathfrak{m}}(M)$ is injective and $H^i_{\mathfrak m} (H^c_I(M))= 0$ for all $i<  d$.
\item[(2)] The natural map $\Ext^{d}_R(k,H^c_I(M))\to \Ext^{t}_R(k, M)$ is injective and $\Ext^i_R(k,H^c_I(M))=0$ for all $i< d$.
\end{itemize}
\end{itemize}
\end{theorem}
Here $D(\cdot)$ denotes the Matlis dual functor. The existence of these natural maps is shown in Propositions \ref{2.4} and \ref{9}. With the additional assumption of $\inj \dim_R(H^c_{I}(M))\leq d$ we succeed to prove that the natural homomorphism $H^{d}_{\mathfrak{m}}(H^c_{I}(M))\to H^t_{\mathfrak{m}}(M)$ is surjective resp. injective (see Theorem \ref{4.4}). This actually generalizes the result proven in case of a complete local Gorenstein ring in \cite[Theorem 4.4]{waqas}. The similar results are obtained for the homomorphism $E\to H^c_{\mathfrak m}(D(H^c_{I}(R)))$ (see Theorem \ref{111}) where $E= E_R(k)$ is the injective hull of $k$.

Another part of our investigation is to derive the natural homomorphism
\[
\Tor^{R}_c(k, H^c_I(M))\to \Tor^{R}_c(k, H^c_J(M))
\]
for two ideals $J\subseteq I$ of $R$ such that $\grade(I,M)= \grade(J,M)= c$ (see Proposition \ref{12}). Note that there is a necessary condition such that this natural homomorphism is an isomorphism. As a consequence of this we are able to prove the following result:
\begin{theorem}\label{d}
Let $M$ be a non-zero finitely generated $R$-module such that $c = \grade(I,M)= \grade(J,M)$ where $J\subseteq I\subseteq R$ are two ideals. Suppose that $\Rad IR_{\mathfrak{p}}= \Rad JR_{\mathfrak{p}}$ for all ${\mathfrak{p}}\in V(J)\cap \Supp_R(M)$ such that $\depth_{R_\mathfrak{p}} (M_{\mathfrak{p}})\leq c$. Assume in addition that $H^i_J(M)= 0$ for all $i\neq c$. Then the natural homomorphism
\[
\Tor^{R}_c(k,H^c_I(M))\to k\otimes_R M
\]
is an isomorphism.
\end{theorem}
Note that the last Theorem \ref{d} is actually the dual statement of one the main results of \cite[Theorem 1.2]{p7} and \cite[Theorem 1.2]{z}.

\section{Preliminaries}
Throughout this paper we will denote $(R,\mathfrak m)$ by a commutative Noetherian local ring with $\dim(R)=n$ and $k= R/{\mathfrak m}$ its residue field. Let $D(\cdot):= \Hom_R(\cdot, E)$ denote the Matlis dual functor, where $E= E_R(k)$ is the injective hull of $k$. For the basic results of homological algebra of complexes we refer to \cite{w} and \cite{har1}. Moreover in the following we make some notation which will be used in the whole article.

\begin{notation}
$(1)$ The symbol $"\cong"$ indicates the isomorphisms between modules.

$(2)$ Let $X\to Y$ be a morphism of complexes of $R$-modules. Then it is called quasi-isomorphism, i.e. homologically isomorphism, if it induces the isomorphism between the homologies of $X$ and $Y$. In this case we will write it as $X \qism Y$.
\end{notation}
Firs of all note that the following version of Local Duality was already derived in \cite[Lemmma 3.1]{waqas2}. For a generalization to arbitrary cohomologically complete intersection ideals we refer to \cite[Theorem 6.4.1]{he}, \cite[Lemma 2.4]{waqas1} or \cite[Theorem 3.1]{p1}. We include it here for our convenience.
\begin{lemma} \label{2.1}
Let $I$ be an ideal of a Cohen-Macaulay ring $R$ with $\dim(R)= n$. Then for any $R$-module $M$ and for all $i\in \mathbb{Z}$ we have
\begin{itemize}
\item[(1)] $\Tor_{n-i}^R(M, H^n_{\mathfrak{m}}(R)) \cong H^i_{\mathfrak{m}}(M)$.
\item[(2)] $D(H^i_\mathfrak{m}(M)) \cong \Ext^{n-i}_R(M, D(H^n_{\mathfrak{m}}(R)))$.
\end{itemize}
\end{lemma}

In the next we need a result which is actually proved in \cite[Proposition 1.4]{pet1} (see also \cite[Lemma 2.5]{zr} for a slight generalization). Moreover the elementary proof of the following Lemma is given in \cite[Proposition 2.6]{waqas2}.
\begin{lemma} \label{2.3}
Suppose that $X$ is an arbitrary $R$-module. Then
$H^i_\mathfrak{m} (X)= 0$ for all $i< s$ if and only if $\Ext^i_R(k,X)= 0$ for all $i< s$ where $s\in \mathbb{N}$. Moreover if one of these equivalent conditions holds, then there is an isomorphism
\[
\Hom_R(k,H^s_\mathfrak{m} (X))\cong \Ext^s_R(k, X).
\]
\end{lemma}

\begin{proof}
See \cite[Proposition 1.4]{pet1}.
\end{proof}

At the end of this section let us recall the definition of the truncation complex. Note that the idea of the truncation complex was firstly given by P. Schenzel (see \cite[Definition 4.1]{p2}). Let $M$ be a finitely generated $R$-module with $\dim_R(M)= t$. Assume that $E^{\cdot}_R(M)$ denotes the minimal injective resolution of $M.$ Then it is easy to see that $\Gamma_{I}(E_R(R/{\mathfrak p}))= E_R(R/{\mathfrak p})$ for all $\mathfrak p\in V(I)$ and zero otherwise. Here $\Gamma_I(-)$ denotes the section functor with support in $I$.

\begin{definition} \label{2.2}
Let $C^{\cdot}_R(I)$ be the cokernel of the embedding of the complexes of $R$-modules
$H^c_I(M)[-c]\to \Gamma_I(E^{\cdot}_R(M))$. This said to be the
truncation complex of $R$ with respect to $I$.
\end{definition}
Note that there ia a short exact sequence of complexes of $R$-modules
\[
0\to H^c_I(M)[-c]\to \Gamma_I(E^{\cdot}_R(M))\to C^{\cdot}_M(I)\to 0.
\]
Then the long exact sequence of cohomologies of this sequence induces that $H^i(C^{\cdot}_M(I))= 0$ for all $i\leq c$ or $i> n$ and $H^i(C^{\cdot}_M(I))\cong H^i_I(M)$ for all $c< i\leq n.$

\section{Natural homomorphisms}
The truncation complex is useful to construct the several natural homomorphisms in the following Proposition. Note that one can also derive these natural homomorphisms as edge homomorphisms of certain spectral sequences. But here for our convenience we will give the elementary proof. Because of we need these construction in the sequel of this paper. Note that the natural homomorphism $\varphi_1$ of the next Proposition is already proved in \cite[Lemma 4.2]{p1}.

\begin{proposition} \label{2.4}
Let $0\neq M$ be a finitely generated $R$-module of $\dim(R)=n$ and $\dim_R(M)= t$. Let $I$ be an ideal with $\grade(I,M)=c$  and $d:= t-c$. Then
\begin{itemize}
\item[(a)] There are the natural homomorphisms:
\begin{itemize}
\item[(1)] $\varphi_1: H^{d}_{\mathfrak{m}}(H^c_{I}(M))\to H^t_{\mathfrak{m}}(M),$
\item[(2)] $\varphi_2: \Ext^{d}_R(k,H^c_I(M))\to \Ext^{t}_R(k, M),$
\item[(3)] $\varphi_3: \Tor^{R}_c(k,H^c_I(M))\to k\otimes_R M.$
\end{itemize}
\item[(b)] Suppose that $R$ is Cohen-Macaulay. Then there are the natural homomorphisms:
\begin{itemize}
\item[(1)] $\varphi_4: \Tor_{n-d}^R(H^n_\mathfrak{m}(R),H^c_I(M))\to H^t_{\mathfrak{m}}(M),$
\item[(2)] $\varphi_5: D(H^t_{\mathfrak{m}}(M))\to \Ext^{n-d}_R(H^c_I(M),D(H^n_\mathfrak{m}(R)))$.
\end{itemize}
\end{itemize}
\end{proposition}

\begin{proof}
Firstly we prove statement in $(a)$. To do this let $\underline{x}= x_1, \ldots ,x_s\in \mathfrak{m}$ with $\Rad \mathfrak{m}= \Rad(\underline{x})R$. We consider the \v{C}ech complex $\Check{C}_{\underline{x}}$ with respect to $\underline{x}$. Now apply $\cdot \otimes_R \Check{C}_{\underline{x}}$ to the short exact sequence of the truncation complex. Then the resulting sequence of complexes remains exact because $\Check{C}_{\underline{x}}$ is a bounded complex of flat $R$-modules. That is there is the following short exact sequence of complexes
\begin{equation}\label{1a}
0\to (\Check{C}_{\underline{x}}\otimes_{R} H^c_{I}(M))[-c]\to \Check{C}_{\underline{x}}\otimes_{R} \Gamma_I(E^{\cdot}_R(M))\to \Check{C}_{\underline{x}}\otimes_{R} C^{\cdot}_M(I)\to 0.
\end{equation}
But the complex in the middle is quasi-isomorphic to the following complex
\[
 \Gamma_{\mathfrak m}(\Gamma_I(E^{\cdot}_R(M)\cong \Gamma_{\mathfrak m}(E^{\cdot}_R(M))
\]
(see \cite[Theorem 3.2]{pet2}). Recall that $\Gamma_I(E^{\cdot}_R(M))$ is a complex of injective $R$-modules and $\Gamma_{\mathfrak m}(\Gamma_I(\cdot))= \Gamma_{\mathfrak m}(\cdot)$. Then the long exact sequence of cohomologies of the last sequence, concentrated in homological degree $t$, induces the following natural map
\[
H^{d}_{\mathfrak{m}}(H^c_{I}(M))\to H^t_{\mathfrak{m}}(M)
\]
where $d:= t-c$. This gives the existence of the natural homomorphism of $(1)$.

Now let $F_{\cdot}^R(k)$ be a free resolution of $k$. Apply the functor $\Hom_R(F_{\cdot}^R(k), .)$ to the short exact sequence of the truncation complex. Since $F_{\cdot}^R(k)$ is a right bounded complex of finitely generated free $R$-modules. Then it induces the following short exact sequences of complexes of $R$-modules
\begin{equation}\label{a}
0\rightarrow \Hom_R(F_{\cdot}^R(k), H^c_I(M))[-c]\rightarrow \Hom_R(F_{\cdot}^R(k), \Gamma_I(E^{\cdot}_R(M)))\rightarrow \Hom_R(F_{\cdot}^R(k), C^{\cdot}_M(I))\rightarrow 0.
\end{equation}
Moreover $\Hom_R(F_{\cdot}^R(k), \Gamma_I(E^{\cdot}_R(M))) \qism \Gamma_I(\Hom_R(F_{\cdot}^R(k), E^{\cdot}_R(M)))$. By \cite[Theorem 1.1]{pet2} the last complex is quasi-isomorphic to $\Check{C}_{\underline{y}}\otimes_{R} \Hom_R(F^{\cdot}_R(k), E^{\cdot}_R(M))$ where $\underline{y}= y_1, \ldots ,y_r\in I$ such that $\Rad I= \Rad(\underline{y})R$. Here we use that any $R$-module of the complex $\Hom_R(F^{\cdot}_R(k), E^{\cdot}_R(M))$ is injective.

Note that there is the following quasi-isomorphism
\[
\Check{C}_{\underline{y}}\otimes_{R} \Hom_R(k, E^{\cdot}_R(M)) \qism \Check{C}_{\underline{y}}\otimes_{R} \Hom_R(F^{\cdot}_R(k), E^{\cdot}_R(M)).
\]
But the complex on the left side is isomorphic to $\Hom_R(k, E^{\cdot}_R(M))$. This is true because of each $R$-module of the complex $\Hom_R(k, E^{\cdot}_R(M))$ has support in $V(\mathfrak{m})$. Therefore our middle complex of the exact sequence \ref{a} is quasi-isomorphic to $\Hom_R(k, E^{\cdot}_R(M))$. Then the long exact sequence of cohomologies of \ref{a} , in degree $t$, induces the required homomorphisms of $(2)$.

Now by tensoring with $F_{\cdot}^R(k)$ to the truncation complex induces
the following exact sequence of complexes of $R$-modules
\begin{equation}\label{11a}
0\to (F_{\cdot}^R(k)\otimes_{R} H^c_I(M))[-c]\to F_{\cdot}^R(k)\otimes_{R}\Gamma_I(E^{\cdot}_R(M))\to F_{\cdot}^R(k)\otimes_{R} C^{\cdot}_M(I)\to 0.
\end{equation}
Because of \cite[Theorem 3.2]{pet2}) we get the quasi-isomorphism of complexes
\[
F_{\cdot}^R(k)\otimes_{R} \Gamma_I(E^{\cdot}_R(M))\qism F_{\cdot}^R(k)\otimes_{R} \Check{C}_{\underline{y}}\otimes_{R} E^{\cdot}_R(M)
\]
Recall that $\Check{C}_{\underline{y}}$ denote the \v{C}ech complex with respect
to $\underline{y}$. Note that we have a quasi-isomorphism $\Check{C}_{\underline{y}}\otimes_R F^{\cdot}_R(k) \qism \Check{C}_{\underline{y}}\otimes_R k$.
Then the last complex is isomorphic to $k$ since $\Supp_R(k)\subseteq V(\mathfrak{m})$. Moreover there is a quasi-isomorphism 
\[
F_{\cdot}^R(k)\otimes_{R} \Check{C}_{\underline{y}}\otimes_{R}M\qism F_{\cdot}^R(k)\otimes_{R} \Check{C}_{\underline{y}}\otimes_{R} E^{\cdot}_R(M)
\]
Let $L_{\cdot}^R$ denote the free resolution of $M$. Then it follows that the morphism of complexes $F_{\cdot}^R(k)\otimes_{R} \Check{C}_{\underline{y}}\otimes_{R}L_{\cdot}^R\to F_{\cdot}^R(k)\otimes_{R} \Check{C}_{\underline{y}}\otimes_{R}M$ induces an isomorphism in homology. By the above remark we get the following quasi-isomorphism
\[
F_{\cdot}^R(k)\otimes_{R} \Check{C}_{\underline{y}}\otimes_{R}L_{\cdot}^R\to k\otimes_{R}L_{\cdot}^R
\]
Hence $H^i(F_{\cdot}^R(k)\otimes_{R} \Check{C}_{\underline{y}}\otimes_{R} E^{\cdot}_R(M))\cong H^i(k\otimes_{R}L_{\cdot}^R)$ for all $i\in \mathbb{Z}$. Then the homology in degree $0$ induces the homomorphism in $(3)$. This finishes the proof of the statement $(a)$.

Finally to prove the assertion of $(b)$ we assume that $R$ is Cohen-Macaulay. Apply the Local Duality Lemma \ref{2.1} (for $M= H^c_I(M)$) we have the isomorphism
\[
H^{d}_{\mathfrak m}(H^c_I(M))\cong \Tor_{n-d}^R(H^c_I(M),H^n_\mathfrak{m}(R)).
\]
Therefore by $(a)$ this gives the homomorphism in $(2)$. By Hom-Tensor Duality the Matlis dual of this last isomorphism induces the following isomorphism
\[
\Ext^{n-d}_R(H^c_I(M),D(H^n_\mathfrak{m}(R)))\cong D(H^d_{\mathfrak m}(H^c_I(M))).
\]
Then the homomorphism in $(2)$ can be easily derived from $(a)$. This completes the proof of the Proposition.
\end{proof}

Note that for any $R$-module $X$ there are the natural homomorphisms
\[
\Ext^i_R(k, X)\to H^i_{\mathfrak m}(X)
\]
for all $i\in \mathbb{N}$. In \cite[Section 4]{hs} M. Hochster has studied about these natural maps in case of canonical modules. In view of the above homomorphism we are interested to relate $\varphi_3$ with the homomorphisms of the next Proposition.

\begin{proposition} \label{9}
With the previous notation there are the following natural homomorphisms:
\[
\varphi_6: \Hom_R(M,k)\to H^c_{\mathfrak m}(D(H^c_I(M))),\text { and }
\]
\[
\varphi_7: D(M)\to H^c_{\mathfrak m}(D(H^c_{I}(M))).
\]
\end{proposition}

\begin{proof}
Since the Matlis dual of $\varphi_3$ induces the following natural homomorphism
\[
\Hom_R(M,k)\to \Ext^c_R(k, D(H^c_I(M))),
\]
because of $D(k\otimes_R M)\cong \Hom_R(M,k)$. Now take composition of this with the homomorphism $\Ext^c_R(k, D(H^c_I(M)))\to H^c_{\mathfrak m}(D(H^c_I(M)))$. Then we get the existence of first natural homomorphism as follows:
\[
\varphi_6: \Hom_R(M,k)\to H^c_{\mathfrak m}(D(H^c_I(M))).
\]
Now we show the existence of $\varphi_7$. To do this let $H = H^c_I(M)$ and $F_{\cdot}(R/\mathfrak m^\alpha)$ denote a minimal free resolution of $R/\mathfrak m^\alpha$ for $\alpha\in \mathbb N$. Then after the implement of the functor $ \cdot \otimes_R F_{\cdot}(R/\mathfrak m^\alpha)$ to the truncation complex gives us the following natural homomorphisms
\[
\Tor^{R}_c(R/\mathfrak m^\alpha, H^c_I(M))\to R/\mathfrak m^\alpha\otimes_R M
\]
for all $\alpha\in \mathbb N$ (see Proposition \ref{2.4} $(a)$). Then the Matlis dual of this induces the following homomorphism
\[
D(R/\mathfrak m^\alpha\otimes_R M)\to \Ext^c_R(R/\mathfrak m^\alpha, D(H^c_I(M)))
\]
for each $\alpha\in \mathbb N$. By passing to the direct limit of this gives rise to the homomorphism
\[
\varinjlim D(R/\mathfrak m^\alpha\otimes_R M)\to H^c_{\mathfrak{m}}(D(H^c_I(M))).
\]
By Hom-Tensor Duality we have $\varinjlim D(R/\mathfrak m^\alpha\otimes_R M) \cong \varinjlim \Hom_R(R/\mathfrak m^\alpha, D(M))$. Moreover support of the module $D(M)$ is contained in $\{\mathfrak{m}\}$. So the last module is isomorphic to $D(M)$. Hence it proves the existence of $\varphi_7$.
\end{proof}

In the next we are interested to characterize the injectivity and surjectivity of all the homomorphisms $\varphi_i,i=1,\dots,7$. In this direction the first result is the following:

\begin{lemma} \label{3.2} Let $R$ be a Cohen-Macaulay ring. With the above notation the following are true:
\begin{itemize}
\item[(1)] $\varphi_1$ is non-zero if and only if $\varphi_4$ is non-zero if and only if $\varphi_5$ is non-zero.
\item[(2)] The following conditions are equivalent:
\begin{itemize}
\item[(i)] $\varphi_1$ is injective (resp. surjective).
\item[(ii)] $\varphi_4$ is injective (resp. surjective).
\item[(iii)] $\varphi_5$ is surjective (resp. injective).
\end{itemize}
\end{itemize}
\end{lemma}

\begin{proof}
Since $R$ is Cohen-Macaulay so by the Local Duality we have the isomorphism
\[
H^d_{\mathfrak m}(H^c_I(M))\cong\Tor_{n-d}^R(H^n_\mathfrak{m}(R),H^c_I(M)).
\]
So $\varphi_1$ is non-zero if and only if $\varphi_4$ is non-zero. By Hom-Tensor Duality it implies that $\varphi_4$ is non-zero if and only if $\varphi_5$ is non-zero. Therefore the statement in $(1)$ is shown to be true. Note that the equivalence of the conditions in the statement $(2)$ can be easily proved by the same arguments.
\end{proof}

In the particular case of a complete Gorenstein ring it was proven in \cite[Theorem 3.2]{waqas} that $\varphi_i,i=1,4,5$, are all non-zero.

The next result tells us that all the homomorphisms of Proposition \ref{2.4} are isomorphisms provided that $M$ is cohomologically complete intersection with respect to $I$. That is $H^i_I(M)= 0$ for all $i\neq c= \grade(I,M)$. Note that our next result is the generalization of \cite[Corollary 4.3]{waqas2}.

\begin{corollary}\label{10}
Let $I$ be an ideal of a local ring $(R,\mathfrak m)$. Suppose that $M$ is a non-zero finitely generated $R$-module with $H^i_I(M)= 0$ for all $i\neq c= \grade(I,M)$. Then the following conditions hold:
\begin{itemize}
\item[(1)] $\varphi_1$ is an isomorphism.
\item[(2)] $\varphi_2$ is an isomorphism.
\item[(3)] $\varphi_3$ is an isomorphism.
\item[(4)] Assume in addition that $R$ is Cohen-Macaulay. Then $\varphi_4$ is an isomorphism.
\item[(5)] $\varphi_5$ is an isomorphism.
\end{itemize}
\end{corollary}
\begin{proof}
Suppose that $E^{\cdot}_R(M)$ is a minimal injective resolution of $M.$ Then $\Gamma_I(E^{\cdot}_R(M))$ is an injective resolution of $H^c_I(M)[-c]$ (because of $H^i_I(M)= 0$ for all $i\neq c$). Then, for the \v{C}ech complex $\Check{C}_{\underline{x}}$ with respect to $\underline{x}$, we have the following quasi isomorphism
\[
(\Check{C}_{\underline{x}}\otimes_{R} H^c_{I}(M))[-c] \qism \Check{C}_{\underline{x}}\otimes_{R} \Gamma_I(E^{\cdot}_R(M)).
\]
Here we assume that $\underline{x}= x_1, \ldots ,x_r\in \mathfrak{m}$ such that $\Rad \mathfrak{m}= \Rad(\underline{x})R$. By the the proof of Proposition \ref{2.4} $(1)$ it follows that the complex on the right side is quasi isomorphic to $\Check{C}_{\underline{x}}\otimes_{R} E^{\cdot}_R(M)$. This proves that $\varphi_1$ is an isomorphism. By Lemma \ref{3.2} the isomorphism in $(1)$ proves the isomorphisms in $(4)$ and $(5)$.

Now for a minimal free resolution $F^{\cdot}_R(k)$ of $k$ we have the following quasi isomorphism
\[
(H^c_I(M)\otimes_R F^{\cdot}_R(k))[-c] \qism \Gamma_I(E^{\cdot}_R(M))\otimes_R F^{\cdot}_R(k).
\]
By the proof of Proposition \ref{2.4} $(3)$ it follows that $H^i(\Gamma_I(E^{\cdot}_R(M))\otimes_R F^{\cdot}_R(k))= 0$ for all $i\neq 0$ and it is $k\otimes_R M$ for $i=0$. This completes the proof of $(3)$.

Since $\Gamma_I(E^{\cdot}_R(M))$ is an injective resolution of $H^c_I(M)[-c]$. It induces the following quasi isomorphism
\[
\Hom_R(F_{\cdot}^R(k), H^c_I(M))[-c] \qism \Hom_R(F_{\cdot}^R(k), \Gamma_I(E^{\cdot}_R(M))).
\]
Again by the proof of Proposition \ref{2.4} $(2)$ it follows that the later complex is quasi isomorphic to $\Hom_R(F_{\cdot}^R(k), E^{\cdot}_R(M))$. Hence we get $\varphi_2$ is an isomorphism.
\end{proof}

\section{The homomorphisms $\varphi_1$ and $\varphi_2$}
In this section we investigate the natural homomorphisms $\varphi_1$ and $\varphi_2$ of Proposition \ref{2.4}. Here we will relate several interpretations of $\varphi_1$ and $\varphi_2$. In particular our first main result is the following:
\begin{theorem} \label{44}
Let $M$ be a non-zero Cohen-Macaulay $R$-module of $\dim_R(M)= t$. Suppose that $\grade(I,M)= c$ for an ideal $I$ and $d: = t-c$. Then the following conditions are equivalent:
\begin{itemize}
\item[(1)] $\varphi_1$ is injective and $H^i_{\mathfrak m} (H^c_I(M))= 0$ for all $i<  d$.
\item[(2)] $\varphi_2$ is injective and $\Ext^i_R(k,H^c_I(M))=0$ for all $i< d$.
\end{itemize}
\end{theorem}

\begin{proof}
Note that the equivalence of both of the vanishing statements in $(1)$ and $(2)$ follows by Lemma \ref{2.3} (for $X= H^c_I(M)$). Then by Proposition \ref{2.4} it induces the commutative diagram
\[
\begin{array}{ccc}
  \Ext^d_R(k, H^c_I(M)) &  \to & H^d_{\mathfrak m}(H^c_I(M)) \\
  \downarrow \varphi_2&   & \downarrow \varphi_1\\
   \Ext^t_R(k, M) & \to &  H^t_{\mathfrak m}(M)
\end{array}
\]
This gives rise to the following commutative diagram
\[
\begin{array}{ccc}
  \Ext^d_R(k, H^c_I(M)) &  \to & \Hom_R(k,H^d_{\mathfrak m}(H^c_I(M))) \\
  \downarrow \varphi_2&   & \downarrow \\
   \Ext^t_R(k, M) & \to &  \Hom_R(k,H^t_{\mathfrak m}(M))
\end{array}
\]
Since $M$ is Cohen-Macaulay so $H^i_{\mathfrak m}(M)= 0$ for all $i\neq t$. Then, by the equivalence of the vanishing statements, both of the horizontal homomorphisms of the last diagram are isomorphisms (see Lemma \ref{2.3}).

We only need to prove the equivalence of the injectivity. For this let us assume that $\varphi_1$ is injective. It implies that $\Hom_R(k,H^d_{\mathfrak m}(H^c_I(M)))\to \Hom_R(k,H^t_{\mathfrak m}(M))$ is injective. Hence, by the above commutative diagram, $\varphi_2$ is injective.

Conversely, suppose that $\varphi_2$ is injective. Then the cohomology sequence of the exact sequence \ref{1a} provides the following exact sequence of $R$-modules
\begin{equation}\label{qw}
0\to H^{t-1}_{\mathfrak m}(C^{\cdot}_M(I))\to H^d_{\mathfrak m}(H^c_I(M))\to H^t_{\mathfrak m} (M)
\end{equation}
To this end recall that $H^i_{\mathfrak m}(M)= 0$ for all $i\neq t$ (since $M$ is Cohen-Macaulay). We will prove that $H^{t-1}_{\mathfrak m}(C^{\cdot}_M(I))= 0$. Apply the functor $\Hom_R(k,\cdot)$ to this exact sequence we get the following exact sequence
\[
0\to \Hom_R(k,H^{t-1}_{\mathfrak m}(C^{\cdot}_M(I)))\to \Hom_R(k,H^d_{\mathfrak m}(H^c_I(M)))\to \Hom_R(k,H^t_{\mathfrak m} (M))
\]
But $\Hom_R(k,H^d_{\mathfrak m}(H^c_I(M)))\to \Hom_R(k,H^t_{\mathfrak m}(M))$ is injective (by the above commutative diagram and the assumption on $\varphi_2$). It follows that $\Hom_R(k,H^{t-1}_{\mathfrak m}(C^{\cdot}_M(I)))= 0$. It is well-known that if $X$ is an $R$-module with support in $V({\mathfrak{m}})$. Then the socle of $X$ is zero if and only if $X$ is zero. This proves that $H^{t-1}_{\mathfrak m}(C^{\cdot}_M(I))= 0$ because of $\Supp(H^{t-1}_{\mathfrak m}(C^{\cdot}_M(I)))\subset \{\mathfrak{m}\}$. Hence
\[
\varphi_1: H^d_{\mathfrak m}(H^c_I(M))\to H^t_{\mathfrak m} (M)
\]
is injective (by sequence \ref{qw}). This completes the proof of the Theorem.
\end{proof}

In the next result we are going to prove the surjectivity of $\varphi_1$. In fact this result can be used to show that the natural homomorphism $\Ext^d_R(k, H^c_I(M))\to \Ext^t_R(k, M)$ is non-zero. Note that the following version of $\varphi_1$ and $\varphi_2$ is already proved in case of a local Gorenstein ring, see \cite[Theorem 4.4]{waqas}. Here we will generalize it to Cohen-Macaulay modules.
\begin{theorem} \label{4.4}
With the previous notation assume that $\varphi_2$ is surjective (resp. injective) and $\Ext^d_R(k, H^c_I(M))= 0$ for all $i> d$. Then $\varphi_1$ is surjective (resp. injective).
\end{theorem}

\begin{proof}
Let $H = H^c_I(M)$ then it is a consequence of the truncation complex that there are the natural homomorphisms
\[
f_\alpha: \Ext^d_R(R/\mathfrak m^\alpha, H)\to \Ext^t_R(R/\mathfrak m^\alpha, M)
\]
for all $\alpha\in \mathbb N$ (see Proposition \ref{2.4} $(2)$). We claim that $f_\alpha$ is surjective (resp. injective) for each $\alpha\in \mathbb N$. Let $\alpha= 1$ then, by assumption on $\varphi_2$, we are true. Now the short exact sequence
\begin{equation}\label{q1}
0\to \mathfrak m^\alpha/\mathfrak m^{\alpha+1}\to R/\mathfrak m^{\alpha+1}\to R/\mathfrak m^{\alpha}\to 0
\end{equation}
induces the following commutative diagram with exact rows
\[
\begin{array}{cccccccc}
  & & \Ext^d_R(R/\mathfrak m^\alpha, H) & \to & \Ext^d_R(R/\mathfrak m^{\alpha+1}, H) & \to & \Ext^d_R(\mathfrak m^{\alpha}/\mathfrak m^{\alpha+1}, H) & \to 0\\
    &   & \downarrow {f_\alpha }&  & \downarrow {f_{\alpha+1}} &   & \downarrow f  &\\
 0 & \to & \Ext^t_R(R/\mathfrak m^\alpha, M) & \to & \Ext^t_R(R/\mathfrak m^{\alpha+1}, M) & \to & \Ext^t_R(\mathfrak m^\alpha/\mathfrak m^{\alpha+1}, M) &
\end{array}
\]
To this end note the above row (resp. the lower row) is right exact (rep. left exact) because of assumption on the vanishing of Ext modules (resp. $\depth_R(M)= t$ see \cite[Theorem 1.2.5]{h}).

Since ${\mathfrak m}^s/{\mathfrak m}^{s+1}$ is a finite dimensional $k$-vector space. Then the natural homomorphism $f$ is surjective (resp. injective) because of $\varphi_2$ is surjective (resp. injective). It implies that $f_{\alpha}$ is surjective (resp. injective) for all $\alpha \in \mathbb{N}$ (by snake lemma and induction hypothesis). It completes the proof of the claim. So the following sequence of direct systems
\[
\{\Ext^d_R(R/\mathfrak m^\alpha, H^c_I(M))\}\mathop\to\limits^{f_\alpha} \{\Ext^t_R(R/\mathfrak m^\alpha, M)\}
\]
is right exact (resp. left exact). Since direct limit is exact functor on direct systems. So by passing to the direct limit it induces that
\[
\varphi_1 : H^d_{\mathfrak m}(H^c_I(M))\to H^t_{\mathfrak{m}}(M)
\]
is surjective (resp. injective).
\end{proof}

\begin{remark}\label{3.1}
Suppose that $0\neq M$ is a Cohen-Macaulay $R$-module with $\inj \dim_R(H^c_I(M))\leq d = t-c$. If $\varphi_2$ is surjective (resp. injective). Then $\varphi_1$ is surjective (resp. injective). It is clear from the proof of the last Theorem \ref{4.4}. The similar result is obtained in \cite[Theorem 4.4]{waqas} for a Gorenstein local ring.
\end{remark}

Note that the following result is a generalization of \cite[Corollary 4.6]{waqas2}. The proof given there is depend on the derived category theory. Here we get the similar result as a consequence of our previous Theorem \ref{4.4}.
\begin{corollary} \label{41}
Fix the notation of Theorem \ref{44}. Then the following conditions hold:
\begin{itemize}
\item[(1)] If $\varphi_2$ is an isomorphism and $\Ext^i_R(k,H^c_I(M))=0$ for all $i\neq d$. Then $\varphi_1$ an isomorphism and $H^i_{\mathfrak m} (H^c_I(M))= 0$ for all $i\neq d$.
\item[(2)] If $\varphi_1$ is an isomorphism and $H^i_{\mathfrak m} (H^c_I(M))= 0$ for all $i< d$. Then $\varphi_2$ is an isomorphism and $\Ext^i_R(k,H^c_I(M))=0$ for all $i< d$. Moreover there are isomorphism
    \[
\Ext^{i}_R(k,H^c_I(M))\rightarrow \Ext^{i+c}_R(k, M)
\]
for all $i> d$.
\end{itemize}
\end{corollary}

\begin{proof}
Firstly we prove the statement $(1)$. By Theorems \ref{4.4} and \ref{44} we only need to prove that $H^i_{\mathfrak m} (H^c_I(M))= 0$ for all $i> d$. To do this note that $\dim_R(H^c_I(M))\leq d=t-c$ (see \cite{k}). Then by Grothendieck's vanishing result, \cite[Theorem 6.1.2]{b}, it follows that $H^i_{\mathfrak m} (H^c_I(M))= 0$ for all $i> d$.

For the statement $(2)$ we only check the isomorphisms (for vanishing result see Lemma \ref{2.3} for $X= H^c_I(M)$ and $s=d$). But the isomorphisms follow from \cite[Lemma 4.5]{waqas2}.
\end{proof}

At the end of this sectoin we will relate the endomorphism rings of $M$ and the local cohomology module $H^c_{I}(M).$ This concept firstly studied by the author and Z. Zahid in \cite[Theorem 1.1]{z}. Before this let us make some identifications. We will denote $\hat{R}^{IR_\mathfrak p}_\mathfrak p$ by the completion of $R_\mathfrak p$ with respect to the ideal $IR_\mathfrak p$ where ${\mathfrak p}\in  V(I)\cap \Supp_R(M)$. Let $k({\mathfrak p})$ stand for the the residue field of $R_\mathfrak p$. Moreover for $c = \grade(I,M)$ we set $h({\mathfrak p}):= \dim(M_{\mathfrak p})- c$.

\begin{corollary} \label{01}
Let $0\neq M$ be a Cohen-Macaulay $R$-module of dimension $\dim_R(M)=t$. Let $I\subseteq R$ be an ideal of $c = \grade(I,M).$ Suppose that for all ${\mathfrak p}\in  V(I)\cap \Supp_R(M)$ the natural homomorphism
\[
\Ext^{h({\mathfrak p})}_{R_{\mathfrak p}}(k({\mathfrak p}),H^c_{IR_{\mathfrak p}}(M_{\mathfrak p}))\rightarrow \Ext^{\dim(M_{\mathfrak p})}_{R_{\mathfrak p}}(k({\mathfrak p}), (M_{\mathfrak p}))
\]
is an isomorphism and $\Ext^i_{R_{\mathfrak p}}(k({\mathfrak p}), H^c_{IR_{\mathfrak p}}(M_{\mathfrak p}))= 0$ for all $i\neq h({\mathfrak p})$. Then the natural homomorphism
\[
\Hom_{\hat{R}^{IR_\mathfrak p}_\mathfrak p}(\hat{M}^{IR_\mathfrak p}_\mathfrak p, \hat{M}^{IR_\mathfrak p}_\mathfrak p)\to \Hom_{R_\mathfrak p}(H^c_{IR_\mathfrak p}(M_\mathfrak p),H^c_{IR_\mathfrak p}(M_\mathfrak p))
\]
is an isomorphism for all ${\mathfrak p}\in  V(I)\cap \Supp_R(M)$.
\end{corollary}

\begin{proof}
We claim that our assumption implies that $H^i_{I}(M)= 0$ for all $i\neq c$. To prove this we will use induction on $\dim_R(M/IM)$. Let $\dim_R(M/IM)= 0$ then it follows that $\Supp(H^i_I(M))\subseteq V({\mathfrak m})$ for all $i\in \mathbb{Z}$. So our assumption is true for $\mathfrak p= \mathfrak m$.

Then by Corollary \ref{41} $\varphi_1$ is an isomorphism and $H^i_{\mathfrak m} (H^c_I(M))= 0$ for all $i\neq d$. This implies, in view of the long exact sequence of cohomologies of sequence \ref{1a}, that $H^i_\mathfrak m(C^{\cdot}_M(I))= 0$ for all $i\in \mathbb Z.$ Recall that $H^i_{\mathfrak m}(M)= 0$ for all $i\neq t$ since $M$ is Cohen-Macaulay.

Since $\Supp_R(H^i(C^{\cdot}_M(I)))\subseteq V(\mathfrak m)$. So by \cite[Lemma 2.5]{waqas2} in view of definition of the truncation complex we have
\[
0=H^i_\mathfrak m(C^{\cdot}_M(I))\cong H^i(C^{\cdot}_M(I))\cong H^i_I(M)
\]
for all $c<i\leq n$. This proves the claim for $\dim(M/IM)= 0$. Now let us assume that $\dim(M/IM)> 0$. Then it is easy to see that $\dim(M_{\mathfrak p}/IM_{\mathfrak p})< \dim(M/IM)$
for all ${\mathfrak p} \in V(I)\cap \Supp_R(M) \setminus \{ \mathfrak m\}$. By the induction hypothesis we have
\[
H^i_{IR_{\mathfrak p}}(M_{\mathfrak p})= 0
\]
for all $i\neq c$ and for all ${\mathfrak p} \in V(I)\cap \Supp(M)\setminus \{ \mathfrak m\}$. That is $\Supp(H^i_I(M))\subseteq V({\mathfrak m})$ for all $i\neq c$. Then by the similar arguments as we use above, for $\mathfrak p= \mathfrak m$, our claim is true. That is $H^i_{I}(M)= 0$ for all $i\neq c$. Then by \cite[Proposition 2.7 and Lemma 4.4]{waqas2} it follows that
\[
H^i_{IR_{\mathfrak p}}(M_{\mathfrak p})= 0
\]
for all $i\neq c=\grade(IR_\mathfrak p, M_\mathfrak p)$ and for all ${\mathfrak p}\in V(I) \cap \Supp_R M$. Now the existence of the natural homomorphism
\[
\Hom_{\hat{R}^{IR_\mathfrak p}_\mathfrak p}(\hat{M}^{IR_\mathfrak p}_\mathfrak p, \hat{M}^{IR_\mathfrak p}_\mathfrak p)\to \Hom_{R_\mathfrak p}(H^c_{IR_\mathfrak p}(M_\mathfrak p),H^c_{IR_\mathfrak p}(M_\mathfrak p))
\]
was shown in \cite[Theorem 1.1]{z}. But $H^i_{IR_{\mathfrak p}}(M_{\mathfrak p})= 0$ for all $i\neq c=\grade(IR_\mathfrak p, M_\mathfrak p)$. Hence it proves our required isomorphism in view of \cite[Theorem 1.1]{z}.
\end{proof}

\begin{remark}\label{3.1}
Note that it is unknown to us whether the last Corollary \label{01} is true for only ${\mathfrak p}= {\mathfrak m}$, the maximal ideal. However it is true if we replace $M$ by a complete local Gorenstein ring (see \cite[Theorem 4.4]{waqas}).
\end{remark}

\section{The homomorphisms $\varphi_3,\varphi_6$ and $\varphi_7$}
In this section our intention is to investigate the homomorphisms $\varphi_3,\varphi_6$ and $\varphi_7$ of Propositions \ref{2.4} and \ref{9} in more details. In fact our investigation is useful to prove that the natural homomorphism $\Tor^{R}_c(k,H^c_I(M))\to k\otimes_R M$ is non-zero. First of all note that the motivation of the next Proposition is the relation between the endomorphism rings of modules $H^c_I(M)$ and $H^c_J(M)$ such that $c = \grade(I,M)= \grade(J,M)$ and $J\subseteq I$. This was firstly introduced by P. Schenzel in case of $M=R$ a local Gorenstein ring (see \cite[Theorem 1.2]{p7}). Moreover for an extension to modules we refer to \cite[Theorem 1.1]{waqas1} and \cite[Theorem 1.2]{z}. Here we are interested to prove the similar result for Tor modules.
\begin{proposition}\label{12}
Let $M$ be a non-zero finitely generated $R$-module such that $c = \grade(I,M)= \grade(J,M)$ where $J\subseteq I$ are two ideals. Then we have the following results:
\begin{itemize}
\item[(1)] There is a natural homomorphism
\[
\Tor^{R}_c(k, H^c_I(M))\to \Tor^{R}_c(k, H^c_J(M)).
\]
\item[(2)] Suppose that $\Rad IR_{\mathfrak{p}}= \Rad JR_{\mathfrak{p}}$ for all ${\mathfrak{p}}\in V(J)\cap \Supp_R(M)$ such that $\depth_{R_\mathfrak{p}} (M_{\mathfrak{p}})\leq c$. Then the above natural homomorphism is an isomorphism.
\end{itemize}
\end{proposition}
\begin{proof}
Since $J\subset I$ it induces the following short exact sequence
\[
0\to I^{\alpha}/J^{\alpha}\to R/J^{\alpha}\to R/I^{\alpha}\to 0
\]
for each integer $\alpha \geq 1$. Let $E^{\cdot}_R(M)$ be a minimal injective resolution of $M.$ Then we have the exact sequence
\[
0\to \Hom_R(R/I^{\alpha}, E^{\cdot}_R(M))\to \Hom_R(R/J^{\alpha}, E^{\cdot}_R(M))\to \Hom_R(I^{\alpha}/J^{\alpha},E^{\cdot}_R(M))\to 0.
\]
Then the cohomology sequence of this, at degree c, gives rise to the exact sequence
\[
0\to \Ext^c_R(R/I^{\alpha}, M)\to \Ext^c_R(R/J^{\alpha}, M)\to \Ext^c_R(I^{\alpha}/J^{\alpha},M).
\]
To this end note that $\grade(I^{\alpha}/J^{\alpha},M)\geq c$ (see \cite[Proposition 2.1]{waqas1}). Now the direct limit is an exact functor. So pass to the direct limit of this sequence we get the exact sequence
\begin{equation}\label{q}
0\to H^c_{I}(M)\to H^c_{J}(M)\mathop\to\limits^f \lim\limits_{\longrightarrow} \Ext^c_R(I^{\alpha}/J^{\alpha},M)
\end{equation}
Let $N:= \im f$. Then this induces the short exact sequence
\[
0\to H^c_{I}(M)\to H^c_{J}(M)\to N\to 0
\]
So for a minimal free resolution $F^{\cdot}_R(k)$ of $k$ we have the following short exact sequence of complexes
\[
0\to H^c_I(M)\otimes_R F^{\cdot}_R(k)\to H^c_J(M)\otimes_R F^{\cdot}_R(k)\to N\otimes_R F^{\cdot}_R(k)\to 0
\]
Hence from the long exact sequens of cohomologies there is the following natural homomorphism
\[
\Tor^{R}_c(k, H^c_I(M))\to \Tor^{R}_c(k, H^c_J(M))
\]
it completes the proof of $(1)$.

For the proof of $(2)$ note that $\Ext^c_R(I^{\alpha}/J^{\alpha},M)= 0$ for all $\alpha\geq 1$ under the additional assumption in the statement $(2)$ (see the proof of \cite[Theorem 4.1(b)]{z}). Hence $(2)$ is true by virtue of the exact sequence \ref{q}.
\end{proof}

Consequently our Proposition \ref{12} gives us a characterization such that our natural homomorphism $\varphi_3$ becomes an isomorphism as follows:
\begin{corollary}\label{0}
Let $0\neq M$ be a finitely generated $R$-module of $\dim_R(M)= t$. Let $J\subseteq I$ be two ideals and $H^i_{J}(M)= 0$ for all $i\neq c= \grade(J,M)= \grade(I,M)$.
Suppose that $\Rad IR_{\mathfrak{p}}= \Rad JR_{\mathfrak{p}}$ for all prime ideals
${\mathfrak{p}}\in V(J)\cap \Supp_R(M)$ with $\depth_{R_\mathfrak{p}} (M_\mathfrak{p})\leq c$. Then the natural homomorphism
\[
\varphi_3: \Tor^{R}_c(k, H^c_I(M))\to k\otimes_R M
\]
is an isomorphism.
\end{corollary}

\begin{proof}
It is a consequence of Proposition \ref{12} and Corollary \ref{10}.
\end{proof}

In the following result we will give a characterization to check that the natural homomorphism $\Tor^{R}_c(k,H^c_I(M))\to k\otimes_R M$ is non-zero. In fact in the parallel of Theorem \ref{44} we are succeeded to derive the equivalence of surjectivity and injectivity of $\varphi_3,\varphi_6$ and $\varphi_7$ in terms of vanishing of Betti numbers of the module $H^c_I(M)$.
\begin{theorem} \label{1.1}
Let $M$ be a non-zero finitely generated module over $R$ and $\dim_R(M)= t$. Suppose that $\grade(I,M)= c$ for an ideal $I$. Then the following conditions are equivalent:
\begin{itemize}
\item[(1)] $\varphi_3$ is surjective and $\Tor_i^R(k,H^c_I(M))=0$ for all $i< c$.
\item[(2)] $\varphi_6$ is injective and $H^i_{\mathfrak m} (D(H^c_I(M)))= 0$ for all $i<  c$.
\item[(3)] $\varphi_7$ is injective and $H^i_{\mathfrak m} (D(H^c_I(M)))= 0$ for all $i<  c$.
\end{itemize}
Moreover if any of the equivalent conditions holds then $H^c_{\mathfrak m} (D(H^c_I(M)))\neq 0$.
\end{theorem}

\begin{proof}
Firstly we prove that the statement in $(1)$ is equivalent to the statement in $(3)$. To do this apply Lemma \ref{2.3} for $X= D(H^c_I(M))$. Then it follows that the vanishing statement in $(1)$ is equivalent to the fact that $\Ext^i_R(k,D(H^c_I(M)))=0$ for all $i< c$. Since $\Ext^i_R(k,D(H^c_I(M)))\cong D(\Tor_i^R(k,H^c_I(M)))$ for all $i\in \mathbb{Z}$ (by Hom-tensor Duality). Therefore it proves the equivalence of the vanishing statements in $(1)$ and $(3)$.

Now by virtue of Propositions \ref{2.4} and \ref{9} the natural homomorphism $M\to k\otimes_R M$ induces the commutative diagram
\[
\begin{array}{ccc}
  D(k\otimes_R M) &  \to & D(M) \\
  \downarrow &   & \downarrow \varphi_7\\
   \Ext^c_R(k, D(H^c_I(M))) & \to &  H^c_{\mathfrak m}(D(H^c_I(M)))
\end{array}
\]
Then applying $\Hom_R(k, \cdot)$ to the above diagram it provides the following commutative diagram
\[
\begin{array}{ccc}
 D(k\otimes_R M) & \to & \Hom_R(k,D(M)) \\
  \downarrow &   & \downarrow \\
    \Ext^c_R(k, D(H^c_I(M))) & \to & \Hom_R(k, H^c_{\mathfrak m}(D(H^c_I(M))))
\end{array}
\]
Then both of the horizontal homomorphisms are isomorphisms (by Hom-Tensor Duality and the equivalence of the vanishing statements, see Lemma \ref{2.3}). Suppose that $\varphi_7$ is injective then the homomorphism $\Hom_R(k,D(M))\to \Hom_R(k, H^c_{\mathfrak m}(D(H^c_I(M))))$ is injective. By view of the commutative diagram $D(k\otimes_R M)\to \Ext^c_R(k, D(H^c_I(M)))$ is injective. But this last homomorphism is the Matlis dual of $\varphi_3$. This proves the surjectivity of $\varphi_3$.

For the converse let $\varphi_3$ be surjective. Then by the above arguments the homomorphism $\Hom_R(k,D(M))\to \Hom_R(k, H^c_{\mathfrak m}(D(H^c_I(M))))$ is injective. Now let $N:=\ker\varphi_7$ then it induces the following exact sequence
\begin{equation}\label{1qw1}
0\to N\to D(M)\to H^c_{\mathfrak m}(D(H^c_I(M)))
\end{equation}
We claim that $N= 0$. Note that the last sequence induces the following exact sequence
\[
0\to \Hom_R(k,N)\to \Hom_R(k,D(M))\to \Hom_R(k,H^c_{\mathfrak m} (D(H^c_I(M))))
\]
Then we have $\Hom_R(k,N)=0$ this proves the claim. Since the support of $N$ is contained in $\{\mathfrak{m}\}$. Then the result follows from sequence \ref{1qw1}. This finishes the proof of $(1)\Leftrightarrow (3).$

Now we prove that $(1)$ is equivalent to $(2)$. Note that the vanishing statement is obvious. Moreover there are natural homomorphisms
\[
D(k\otimes_R M)\to D(M) \text { and } D(M)\to H^c_{\mathfrak{m}}(D(H^c_I(M))).
\]
But the first map is injective and $\Hom_R(M,k)\cong D(k\otimes_R M)$. So by the equivalence of the statements $(1)$ and $(2)$ it follows that $\varphi_6$ is injective if $\varphi_3$ is surjective.

Conversely suppose $\varphi_6$ is injective. Then by Propositions \ref{2.4} and \ref{9} we have the commutative diagram
\[
\begin{array}{ccc}
  D(k\otimes_R M) &  = & D(k\otimes_R M) \\
  \downarrow &   & \downarrow \varphi_6\\
   \Ext^c_R(k, D(H^c_I(M))) & \to &  H^c_{\mathfrak m}(D(H^c_I(M)))
\end{array}
\]
To this end note that $D(k\otimes_R M)\cong \Hom_R(M,k)$. Then it induces the following commutative diagram
\[
\begin{array}{ccc}
 D(k\otimes_R M) & \to & \Hom_R(k,D(k\otimes_R M)) \\
  \downarrow &   & \downarrow \\
    \Ext^c_R(k, D(H^c_I(M))) & \to & \Hom_R(k, H^c_{\mathfrak m}(D(H^c_I(M))))
\end{array}
\]
Then by the similar arguments as we used above one can easily prove that $\varphi_3$ is surjective. Moreover $H^c_{\mathfrak m} (D(H^c_I(M)))\neq 0$ by definition of $\varphi_7$ in view of the equivalence of the above statements. Recall that $M\neq 0$. This finishes the proof of the Theorem.
\end{proof}

For the next Theorem we will make some more notations. Note that in case of $M=R$ we have the following homomorphisms:
\[
\psi_1:\Tor^{R}_c(k, H^c_I(M))\to k, \text{ and }
\]
\[
\psi_2:E\to H^c_{\mathfrak m} (D(H^c_I(R)))
\]
\begin{theorem}\label{111}
Let $I$ be an ideal of $R$ with $\grade(I)=c$. If the homomorphism $\psi_1$ is injective and $\Tor_i^R(k,H^c_I(R))=0$ for all $i< c$. Then the homomorphism $\psi_2$ is surjective and $H^i_{\mathfrak m} (D(H^c_I(R)))= 0$ for all $i<  c$.
\end{theorem}

\begin{proof}
By Proposition \ref{9} there are the natural homomorphisms
\[
f_\alpha: \Tor^{R}_c(R/\mathfrak m^\alpha, H^c_I(R))\to R/\mathfrak m^\alpha
\]
for all $\alpha\in \mathbb N$. Moreover after the application of the functor $\Tor^{R}_i(\cdot, H^c_I(R))$ to the exact sequence \ref{q1} we get the exact sequence
\[
\Tor^{R}_i(\mathfrak m^\alpha/\mathfrak m^{\alpha+1}, H^c_I(R)) \to \Tor^{R}_i(R/\mathfrak m^{\alpha+1}, H^c_I(R))\to \Tor^{R}_i(R/\mathfrak m^\alpha, H^c_I(R))
\]
for all $i\in \mathbb Z$. Then by induction on $\alpha$ , in view of vanishing of Tor modules, this proves that $\Tor^{R}_i(R/\mathfrak m^\alpha, H^c_I(R)))=0$ for all $i< c$ and for all $\alpha\in \mathbb N$.

Now we show that $f_{\alpha}$ is injective for all $\alpha\in \mathbb N$. Clearly $f_1=\psi_1$ is injective. Then the short exact sequence \ref{q1} induces the following commutative diagram with exact rows
\[
\begin{array}{cccccccc}
  & & \Tor^{R}_c(\mathfrak m^\alpha/\mathfrak m^{\alpha+1}, H^c_I(M)) & \to & \Tor^{R}_c(R/\mathfrak m^{\alpha+1}, H^c_I(M)) & \to & \Tor^{R}_c(R/\mathfrak m^\alpha, H^c_I(M)) & \to 0\\
    &   & \downarrow f&  & \downarrow {f_{\alpha+1}} &   & \downarrow {f_\alpha }  &\\
 0 & \to & \mathfrak m^\alpha/\mathfrak m^{\alpha+1} & \to & R/\mathfrak m^{\alpha+1} & \to & R/\mathfrak m^\alpha & \to 0
\end{array}
\]
Note that the above row is exact because of $\Tor^{R}_i(R/\mathfrak m^\alpha, H^c_I(R))=0$ for all $i< c$ and for all $\alpha\in \mathbb N$. Then the natural homomorphism $f$ is injective because of $\mathfrak m^\alpha/\mathfrak m^{\alpha+1}$ is a finite dimensional $k$-vector space. Hence by induction, in view of snake lemma, it implies that $f_{\alpha}$ is injective for all $\alpha \in \mathbb{N}$. Take the Matlis dual of $f_{\alpha}$ it induces that the surjective homomorphism
\[
D(R/\mathfrak m^\alpha)\to \Ext^c_R(R/\mathfrak m^\alpha, D(H^c_I(R)))
\]
for all $\alpha\in \mathbb N$. Now apply the direct limit to these maps. Since the direct limit is an exact functor so it implies that the homomorphism
\[
\varinjlim D(R/\mathfrak m^\alpha)\to H^c_{\mathfrak{m}}(D(H^c_I(R)))
\]
is surjective. Moreover there is an isomorphism
\[
\varinjlim D(R/\mathfrak m^\alpha)\cong H^0_\mathfrak{m}(E).
 \]
But $H^0_\mathfrak{m}(E)\cong E$ because of $\Supp_R(E)\subseteq \{\mathfrak{m}\}$. Therefore it proves that $\psi_2$ is surjective. Moreover Theorem \ref{1.1} implies that $H^i_{\mathfrak m} (D(H^c_I(R)))= 0$ for all $i<  c$.
\end{proof}

\begin{corollary}\label{0010}
With the previous notation if the homomorphism $\psi_1$ is an isomorphism and $\Tor_i^R(k,H^c_I(R))=0$ for all $i< c$. Then the homomorphism $\psi_2$ is an isomorphism and $H^i_{\mathfrak m} (D(H^c_I(R)))= 0$ for all $i<  c$.
\end{corollary}

\begin{proof}
This is immediately follows from Theorem \ref{1.1} and \ref{111}.
\end{proof}


\begin{thebibliography}{9999}
\bibitem [1] {b} {\sc Brodmann, Sharp:} Local Cohomology. An Algebraic Introduction with Geometric Applications. Cambridge Studies in Advanced Mathematics No. 60. Cambridge University Press, (1998).

\bibitem [2] {h} {\sc W. Bruns, J. Herzog:} Cohen-Macaulay Rings, Cambridge Univ. Press, 39(1998).

\bibitem [3] {k} {\sc K. Divaani-Aazar, R. Naghipour, M. Tousi:} Cohomological dimension of certain algebraic varieties, Proc. Amer. Math. Soc. 130 (2002), 3537-3544.

\bibitem [4] {f1} {\sc H.-B. Foxby:}  Isomorphisms Between Complexes with Applications The Homological Theory of Modules. Math Scand. 40(1977), 5-19.

\bibitem [5] {goth} {\sc A. Grothendieck:} Local Cohomology, Notes by R. Hartshorne, Lecture Notes in Math. vol.41, Springer,1967.

\bibitem [6] {har1} {\sc R. Hartshorne:}  Residues and duality (Lecture Note in Mathematics, Vol. 20, Springer, 1966.

\bibitem [7] {hs} M. Hochster: Canonical Elements in Local Cohomology and the Direct Summand Conjecture, J. Algebra, 84(1983), 503-553

\bibitem [8] {he} {\sc M. Hellus:} Local Cohomology and Matils Duality, ArXiv:math/0703124v1 [math.AC](2007).

\bibitem [9] {he1} {\sc M. Hellus, J. St\"uckrad:} On Endomorphism Rings of Local Cohomology Modules. Proc. Am. Math. Soc. 136(2008), 2333-2341.

\bibitem [10] {pet1} {\sc M. Hellus, P. Schenzel:} On Cohomologically Complete Intersections, J. Algebra, 320(2008), 3733-3748.

\bibitem [11] {p1} {\sc M. Hellus, P. Schenzel:} A Note on Local Cohomology and Duality, J. Algebra, 401(2014),48-61.

\bibitem[12]{waqas} {\sc W. Mahmood, P. Schenzel:} On Invariants and Endomorphism Rings of Certain Local Cohomology Modules, J. Algebra, 372(2012), 56-67.

\bibitem[13]{waqas1} {\sc W. Mahmood:} On Endomorphism Rings of Local Cohomology Modules, ArXiv:1308.2584v1 [math.AC](2013).

\bibitem[14]{waqas2} {\sc W. Mahmood:} On Cohomologically Complete Intersections in Cohen-Macaulay Rings,  ArXiv:1312.6961[math.AC](2013).

\bibitem[15]{z} {\sc W. Mahmood, Z. Zahid:} A note on Endomorphisms of Local Cohomology Modules,  ArXiv:1405.1249 [math.AC](2014).

\bibitem [16] {pet2} {\sc P. Schenzel:} Proregular Sequences, Local Cohomology, and Completion, Math. Scand. 92(2003), 161-180.

\bibitem[17] {p2} {\sc P. Schenzel:}  On Birational Macaulayfications and Cohen-Macaulay Canonical Modules. J. Algebra, 275(2004), 751-770.

\bibitem[18]{p7} {\sc P. Schenzel:} On Endomorphism Rings and Dimensions of Local Cohomology Modules. Proc. Amer. Math. Soc. 137(2009) 1315-1322.

\bibitem[19]{p3} {\sc P. Schenzel:} Matlis Duals of Local Cohomology Modules and their Endomorphism Rings. Arch. Math. 95(2010) 115-123.

\bibitem[20]{p5} {\sc P. Schenzel:} On the Structure of the Endomorphism Ring of a Certain Local Cohomology Module. J. Algebra, 344 (2011), 229-245.

\bibitem [21] {w} {\sc C. Weibel:} An introduction to Homological Algebra, Cambridge Univ. Press, 1994.

\bibitem [22] {zr} {\sc M. R. Zargar:} On a Duality of Local Cohomology Modules of Relative Cohen-Macaulay Rings, ArXiv:1308.3071 [math.AC](2013).

\end{thebibliography}
\end{document}